\documentclass[oneside,reqno,11pt]{amsart}

\usepackage{geometry}
\usepackage{mathtools}
\usepackage{amsmath}
\usepackage{amsthm}
\usepackage{amssymb}
\usepackage{tikz}
\usepackage{tikz-cd}
\usepackage[all,cmtip]{xy}
\usepackage{float}
\usepackage{indentfirst}
\usepackage[mathscr]{euscript}
\usepackage{stmaryrd}
\usepackage{hyperref}

\makeatletter
\def\blfootnote{\gdef\@thefnmark{}\@footnotetext}
\makeatother

\theoremstyle{plain}
\newtheorem{thm}{Theorem}[section]
\newtheorem{prop}[thm]{Proposition}
\newtheorem{lem}[thm]{Lemma}
\newtheorem{cor}[thm]{Corollary}

\theoremstyle{definition}
\newtheorem{dfn}[thm]{Definition}
\newtheorem{exa}[thm]{Example}
\newtheorem*{ack}{Acknowledgements}

\theoremstyle{remark}
\newtheorem{rmk}[thm]{Remark}

\numberwithin{equation}{section}

\title{Systolic inequalities, Ginzburg dg algebras and Milnor fibers}
\author{Jongmyeong Kim}
\address{Center for Geometry and Physics, Institute for Basic Science (IBS), Pohang 37673, Republic of Korea}
\email{myeong@ibs.re.kr}

\begin{document}

\begin{abstract}
We prove categorical systolic inequalities for the derived categories of 2-Calabi--Yau Ginzburg dg algebras associated to ADE quivers and explore their symplecto-geometric aspects.
\end{abstract}

\maketitle


\section{Introduction}

The {\em systole} $\mathrm{sys}(X,g)$ of a Riemannian manifold $(X,g)$ is defined to be the smallest length of non-contractible loops in $(X,g)$.
Loewner showed that, for the 2-torus $T^2$, the inequality
\begin{equation}\label{eq1.1}
\mathrm{sys}(T^2,g)^2 \leq \frac{2}{\sqrt{3}} \mathrm{vol}(T^2,g)
\end{equation}
holds for every metric $g$ on $T^2$.

This inequality, known as the {\em systolic inequality} for $T^2$, can be reinterpreted as follows \cite{Fan}.
First, let us write $T^2_\tau = \mathbb{C}/(\mathbb{Z}+\tau\mathbb{Z})$ for an element $\tau$ of the upper half plane and equip $T^2_\tau$ with the symplectic form $\omega = \frac{i}{2} dz \wedge d\overline{z}$ and the holomorphic volume form $\Omega = dz$.
With respect to these structures, the special Lagrangian submanifolds in $T^2_\tau$ are those come from the straight lines in $\mathbb{C}$ and they coincide with the shortest non-contractible loops in $T^2_\tau$.
Therefore, the inequality \eqref{eq1.1} can be rewritten as
\begin{equation}\label{eq1.2}
\inf \left\{\left. \left\lvert \int_L \Omega \right\rvert \,\right|\, L \text{ is a special Lagrangian submanifold in } T^2_\tau \right\}^2 \leq \frac{1}{\sqrt{3}} \left\lvert \int_{T^2_\tau} \Omega \wedge \overline{\Omega} \right\rvert
\end{equation}
for every $\tau$.

There is also a categorical interpretation of this inequality due to Fan \cite{Fan} motivated by a conjecture of Bridgeland \cite{Bri2} and Joyce \cite{Joy}.
For a Calabi--Yau manifold $(X,\omega,\Omega)$, the conjecture asserts that the holomorphic volume form $\Omega$ should correspond to a stability condition $\sigma = (Z,\mathcal{P})$ on the derived Fukaya category $D^\pi Fuk(X,\omega)$ (more precisely, the conjecture says that the complex moduli of $X$ can be embedded into a quotient of the space of stability conditions) and, under this correspondence, the $\sigma$-semistable objects (resp. the central charge $Z$) should correspond to the special Lagrangian submanifolds in $(X,\omega)$ (resp. the period integral $\int_L \Omega$).

In view of this conjecture, Fan \cite{Fan} defined the {\em categorical systole} $\mathrm{sys}(\sigma)$ of a stability condition $\sigma = (Z,\mathcal{P})$ of a triangulated category $\mathcal{D}$ to be
\begin{equation*}
\mathrm{sys}(\sigma) = \inf\{\lvert Z(E) \rvert \,|\, E \text{ is a $\sigma$-stable object of } \mathcal{D}\}
\end{equation*}
(Definition \ref{dfn2.5}).
There is also a notion of the {\em categorical volume} $\mathrm{vol}(\sigma)$ of a stability condition $\sigma$ defined by Fan--Kanazawa--Yau \cite{FKY} (Definition \ref{dfn2.7}).
Fan \cite[Theorem 3.1]{Fan} then showed the following categorical analogue of the inequality \eqref{eq1.2}:
\begin{equation*}
\mathrm{sys}(\sigma)^2 \leq \frac{1}{\sqrt{3}} \mathrm{vol}(\sigma)
\end{equation*}
for every stability condition $\sigma$ of $D^\pi Fuk(T^2,\omega)$.

In this paper, we will show categorical inequalities for the derived category $\mathcal{D}_Q$ of {\em 2-Calabi--Yau Ginzburg dg algebras} associated to ADE quivers $Q$.

\begin{thm}\label{thm1.1}
Let $Q$ be an ADE quiver with $n$ vertices.
Then, for every $\sigma \in \mathrm{Stab}^\circ(\mathcal{D}_Q)$,
\begin{equation}\label{eq1.3}
\mathrm{sys}(\sigma)^2 \leq \frac{h_Q}{n} \mathrm{vol}(\sigma)
\end{equation}
where $h_Q$ is the Coxeter number of the underlying graph of $Q$.
\begin{center}
\begin{tabular}{c|ccccc}
$Q$ & $A_n$ & $D_n$ & $E_6$ & $E_7$ & $E_8$\\
\hline
$h_Q$ & $n+1$ & $2(n-1)$ & $12$ & $18$ & $30$
\end{tabular}
\end{center}
\end{thm}

We will also see a symplecto-geometric interpretation of Theorem \ref{thm1.1} for $Q = A_n$.
For that case, $\mathcal{D}_Q$ is equivalent to the derived Fukaya category $D^\pi Fuk(X_n)$ of the {\em Milnor fiber} $X_n$ of type $A_n$ \cite{Tho}.
Thus, by the Bridgeland--Joyce's conjecture, the inequality \eqref{eq1.3} should be described in terms of symplectic and complex geometry.

Now let $\mathcal{P}_n$ be the space of polynomials $z^{n+1} + a_1z^{n-1} + a_2z^{n-2} + \cdots + a_n \in \mathbb{C}[z]$ with only simple zeros.
A polynomial $p \in \mathcal{P}_n$ gives rise to a holomorphic volume form $\Omega_p$ on $X_n$.
Then, using $\Omega_p$ and special Lagrangian submanifolds in $X_n$ with respect to it, we can define the systole $\mathrm{sys}(\Omega_p)$ and the volume $\mathrm{vol}(\Omega_p)$.

\begin{thm}\label{thm1.2}
For every $p \in \mathcal{P}_n$,
\begin{equation*}
\mathrm{sys}(\Omega_p)^2 \leq \frac{n+1}{n} \mathrm{vol}(\Omega_p).
\end{equation*}
\end{thm}

In Section \ref{sec2}, we review definitions and basic properties of stability conditions, categorical systole and volume.
In Section \ref{sec3}, we first recall the definition of Ginzburg dg algebras and then prove Theorem \ref{thm1.1}.
The main step of the proof is the calculation of the categorical volume (Proposition \ref{prop3.8}).
Finally, in Section \ref{sec4}, we prove Theorem \ref{thm1.2} which is a geometric counterpart of Theorem \ref{thm1.1} for $Q = A_n$ case.

\begin{ack}
The author thanks Yong-Geun Oh for many valuable discussions especially on special Lagrangian submanifolds in the Milnor fiber of type $A_n$ and Kyeong-Dong Park for carefully reading an earlier draft of this paper.
This work was supported by IBS-R003-D1.
\end{ack}

\section{Preliminaries}\label{sec2}

\subsection{Stability conditions}

Let $\mathcal{D}$ be a $\mathbb{K}$-linear triangulated category.
Throughout this section, we assume $\mathcal{D}$ is of {\em finite type}, i.e., $\mathrm{dim}\, \mathrm{Hom}^*(E,F) < \infty$ for all $E,F \in \mathcal{D}$.
The {\em Grothendieck group} $K(\mathcal{D})$ of $\mathcal{D}$ is the abelian group generated by the objects of $\mathcal{D}$ with the relation $E-F+G$ whenever there is an exact triangle $E \to F \to G \to E[1]$ in $\mathcal{D}$.
For $E,F \in \mathcal{D}$, their {\em Euler form} is defined by
\begin{equation*}
\chi(E,F) = \sum_{i \in \mathbb{Z}} (-1)^i \mathrm{dim}\, \mathrm{Hom}(E,F[i]).
\end{equation*}
Note that it descends to the map $\chi : K(\mathcal{D}) \times K(\mathcal{D}) \to \mathbb{Z}$ for which we use the same notation.
We then define the {\em numerical Grothendieck group} by
\begin{equation*}
\mathcal{N}(\mathcal{D}) = K(\mathcal{D})/\langle E \in K(\mathcal{D}) \,|\, \chi(E,-) = 0 \rangle.
\end{equation*}
In what follows, we will also assume that $\mathrm{rk}\, \mathcal{N}(\mathcal{D}) < \infty$.

\begin{dfn}[\cite{Bri1}]\label{dfn2.1}
A {\em (numerical) stability condition} $\sigma = (Z,\mathcal{P})$ on a triangulated category $\mathcal{D}$ consists of a group homomorphism $Z : \mathcal{N}(\mathcal{D}) \to \mathbb{C}$ and a full additive subcategory $\mathcal{P}(\phi) \subset \mathcal{D}$ for each $\phi \in \mathbb{R}$ which satisfy the following conditions:
\begin{itemize}
\item[(1)] if $0 \neq E \in \mathcal{P}(\phi)$, then $Z(E) \in \mathbb{R}_{>0} e^{i\pi\phi}$;
\item[(2)] $\mathcal{P}(\phi + 1) = \mathcal{P}(\phi)[1]$ for every $\phi \in \mathbb{R}$;
\item[(3)] if $\phi_1 > \phi_2$ and $E_i \in \mathcal{P}(\phi_i)$, then $\mathrm{Hom}(E_1,E_2) = 0$;
\item[(4)] for every $0 \neq E \in \mathcal{D}$, there exist real numbers $\phi_1 > \cdots > \phi_k$ and $E_i \in \mathcal{P}(\phi_i)$ which fit into an iterated exact triangle of the form (which is necessarily unique)
\begin{equation*}
\xymatrix{
{0} \ar[rr] && {*} \ar[rr] \ar[dl] && {*} \ar[dl] & \cdots & {*} \ar[rr] && {E} \ar[dl]\\
& {E_1} \ar[ul]^{+1} && {E_2} \ar[ul]^{+1} &&&& {E_k} \ar[ul]^{+1} &
};
\end{equation*}
\item[(5)] there exists a constant $C > 0$ and a norm $\lVert - \rVert$ on $\mathcal{N}(\mathcal{D}) \otimes_\mathbb{Z} \mathbb{R}$ such that
\begin{equation*}
\lVert E \rVert \leq C \lvert Z(E) \rvert
\end{equation*}
for any $E \in \mathcal{P}(\phi)$ and $\phi \in \mathbb{R}$.
\end{itemize}
\end{dfn}

From the conditions of the above definition, it follows that the full subcategories $\mathcal{P}(\phi)$ are abelian.
For a stability condition $\sigma = (Z,\mathcal{P}) \in \mathrm{Stab}(\mathcal{D})$, we call $Z$ the {\em central charge} and an object (resp. a simple object) in the abelian category $\mathcal{P}(\phi)$ to be {\em ($\sigma$-)semistable} (resp. {\em ($\sigma$-)stable}) of phase $\phi$.

Let us denote by $\mathrm{Stab}(\mathcal{D})$ the space of (numerical) stability conditions on $\mathcal{D}$.
Bridgeland \cite{Bri1} introduced a nice topology on $\mathrm{Stab}(\mathcal{D})$ with respect to which the projection
\begin{equation*}
\mathrm{Stab}(\mathcal{D}) \to \mathrm{Hom}(\mathcal{N}(\mathcal{D}),\mathbb{C});\; (Z,\mathcal{P}) \mapsto Z
\end{equation*}
becomes a local homeomorphism.
Thus the standard complex structure on $\mathrm{Hom}(\mathcal{N}(\mathcal{D}),\mathbb{C}) \cong \mathbb{C}^{\mathrm{rk}\, \mathcal{N}(\mathcal{D})}$ induces the one on $\mathrm{Stab}(\mathcal{D})$.

There are two natural actions on $\mathrm{Stab}(\mathcal{D})$.
The first one is the action of the group $\mathrm{Auteq}(\mathcal{D})$ of exact autoequivalences of $\mathcal{D}$.
For $\Phi \in \mathrm{Auteq}(\mathcal{D})$ and $(Z,\mathcal{P}) \in \mathrm{Stab}(\mathcal{D})$, it is given by
\begin{equation*}
\Phi \cdot (Z,\mathcal{P}) = (Z \circ \Phi^{-1},\Phi(\mathcal{P}))
\end{equation*}
where $\Phi(\mathcal{P})(\phi) = \Phi(\mathcal{P}(\phi))$.
The second one is the action of $\mathbb{C}$.
For $\zeta \in \mathbb{C}$ and $(Z,\mathcal{P}) \in \mathrm{Stab}(\mathcal{D})$, it is given by
\begin{equation*}
(Z,\mathcal{P}) \cdot \zeta = (e^{-i\pi\zeta} Z,\mathcal{P}[\mathrm{Re}(\zeta)])
\end{equation*}
where $\mathcal{P}[\mathrm{Re}(\zeta)](\phi) = \mathcal{P}(\phi + \mathrm{Re}(\zeta))$.

\subsection{Hearts and simple tilting}

For the sake of simplicity, let us assume that $\mathcal{N}(\mathcal{D}) = K(\mathcal{D})$, i.e., the Euler form is non-degenerate, throughout this subsection.

Let $\sigma = (Z,\mathcal{P}) \in \mathrm{Stab}(\mathcal{D})$ be a stability condition.
For an interval $I \subset \mathbb{R}$, define $\mathcal{P}(I) \subset \mathcal{D}$ to be the smallest full extension closed subcategory containing $\mathcal{P}(\phi)$ for all $\phi \in I$.
Then one can show that $\mathcal{D}^{\leq 0} _\sigma= \mathcal{P}(0,\infty)$ is a bounded t-structure on $\mathcal{D}$.
We call the heart $\mathcal{A}_\sigma = \mathcal{P}(0,1]$ of this bounded t-structure the {\em heart} of the stability condition $\sigma$.

Let $\mathcal{A}$ be an abelian category.
A {\em stability function} on $\mathcal{A}$ is a group homomorphism $Z : K(\mathcal{A}) \to \mathbb{C}$ such that for every $0 \neq E \in \mathcal{A}$, $Z(E)$ lies in
\begin{equation*}
\mathbb{H} = \{ r e^{i\pi\phi} \in \mathbb{C} \,|\, r \in \mathbb{R}_{> 0} \text{ and } \phi \in (0,1] \}.
\end{equation*}
For a stability function $Z : K(\mathcal{A}) \to \mathbb{C}$ and $0 \neq E \in \mathcal{A}$, define the {\em phase} of $E$ as
\begin{equation*}
\phi(E) = \frac{1}{\pi} \mathrm{arg}\, Z(E) \in (0,1].
\end{equation*}
Then we call an object $0 \neq E \in \mathcal{A}$ to be {\em ($Z$-)semistable} (resp. {\em ($Z$-)stable}) if for every $0 \neq F \subsetneq E$ one has $\phi(F) \leq \phi(E)$ (resp. $\phi(F) < \phi(E)$).
If $\mathcal{A}$ has a property analogous to Definition \ref{dfn2.1} (4) with respect to $Z$-semistable objects, we say $Z$ has the {\em Harder--Narasimhan property}.

A stability condition can be thought of as a refinement of the heart of a bounded t-structure in the following sense.

\begin{prop}[{\cite[Proposition 5.3]{Bri1}}]
To give a stability condition on $\mathcal{D}$ is equivalent to give a bounded t-structure on $\mathcal{D}$ and a stability function $Z$ on its heart with the Harder--Narasimhan property.
\end{prop}

\begin{proof}
For a stability condition $\sigma = (Z,\mathcal{P}) \in \mathrm{Stab}(\mathcal{D})$, the corresponding t-structure and stability function are given by $\mathcal{D}^{\leq 0} _\sigma= \mathcal{P}(0,\infty)$ and $Z : K(\mathcal{A}_\sigma) \cong K(\mathcal{D}) \to \mathbb{C}$.
For details, see \cite[Proposition 5.3]{Bri1}.
\end{proof}

We denote by $\mathrm{Stab}_\mathcal{A}(\mathcal{D}) \subset \mathrm{Stab}(\mathcal{D})$ the space of stability conditions on $\mathcal{D}$ whose heart is $\mathcal{A}$.
If $\mathcal{A}$ has a finiteness property, $\mathrm{Stab}_\mathcal{A}(\mathcal{D})$ can be easily described.

\begin{lem}[{\cite[Lemma 5.2]{Bri2}}]
Let $\mathcal{A}$ be the heart of a bounded t-structure on $\mathcal{D}$.
Suppose $\mathcal{A}$ is a finite length abelian category with $n$ simple objects $S_1,\dots,S_n$.
Then $\mathrm{Stab}_\mathcal{A}(\mathcal{D})$ is isomorphic to $\mathbb{H}^n$.
\end{lem}

\begin{proof}
The isomorphism is given by $\mathrm{Stab}_\mathcal{A}(\mathcal{D}) \ni \sigma = (Z,\mathcal{P}) \mapsto (Z(S_i))_{i=1}^n \in \mathbb{H}^n$.
For details, see \cite[Lemma 5.2]{Bri2}.
\end{proof}

Let $\mathcal{A}$ be the heart of a bounded t-structure on $\mathcal{D}$ and $(\mathcal{T},\mathcal{F})$ be a torsion pair on it.
Define a full additive subcategory $\mathcal{A}^\sharp$ (resp. $\mathcal{A}^\flat$) $\subset \mathcal{D}$ whose objects are those $E \in \mathcal{D}$ such that
\begin{gather*}
H^{-1}(E) \in \mathcal{F}, H^0(E) \in \mathcal{T} \text{ and } H^i(E) = 0 \text{ for all } i \neq -1,0\\
\text{(resp. } H^0(E) \in \mathcal{F}, H^1(E) \in \mathcal{T} \text{ and } H^i(E) = 0 \text{ for all } i \neq 0,1 \text{)}
\end{gather*}
where $H^i$ denotes the $i$th cohomology with respect to the bounded t-structure corresponding to $\mathcal{A}$.
Then $\mathcal{A}^\sharp$ (resp. $\mathcal{A}^\flat$) is again the heart of a bounded t-structure with the torsion pair $(\mathcal{F}[1],\mathcal{T})$ (resp. $(\mathcal{F},\mathcal{T}[-1])$) \cite{HRS}.
We call $\mathcal{A}^\sharp$ (resp. $\mathcal{A}^\flat$) the {\em forward tilt} (resp. {\em backward tilt}) of $\mathcal{A}$ with respect to $(\mathcal{T},\mathcal{F})$.
In the case that $\mathcal{F}$ (resp. $\mathcal{T}$) is generated by a single simple object $S \in \mathcal{A}$, the corresponding forward tilt (resp. backward tilt) is called {\em simple} and denoted by $\mathcal{A}^\sharp_S$ (resp. $\mathcal{A}^\flat_S$).

Denote by $\mathrm{Sim}(\mathcal{A})$ the set of simple objects in an abelian category $\mathcal{A}$.

\begin{prop}[{\cite[Proposition 5.4]{KQ}}]\label{prop2.4}
Let $\mathcal{A}$ be the heart of a bounded t-structure on $\mathcal{D}$ which is of finite length and $S \in \mathcal{A}$ be a rigid simple object.
Then,
\begin{gather*}
\mathrm{Sim}(\mathcal{A}^\sharp_S) = \{S[1]\} \cup \{\Phi^\sharp_S(M) \,|\, S \neq M \in \mathrm{Sim}(\mathcal{A})\},\\
\mathrm{Sim}(\mathcal{A}^\flat_S) = \{S[-1]\} \cup \{\Phi^\flat_S(M) \,|\, S \neq M \in \mathrm{Sim}(\mathcal{A})\}
\end{gather*}
where
\begin{gather*}
\Phi^\sharp_S(M) = \mathrm{Cone}(M \to \mathrm{Hom}(M,S[1])^\vee \otimes S[1])[-1],\\
\Phi^\flat_S(M) = \mathrm{Cone}(\mathrm{Hom}(S[-1],M) \otimes S[-1] \to M).
\end{gather*}
In particular, $\mathcal{A}^\sharp_S,\mathcal{A}^\flat_S$ are again of finite length with $\lvert\mathrm{Sim}(\mathcal{A}^\sharp_S)\rvert = \lvert\mathrm{Sim}(\mathcal{A}^\flat_S)\rvert = \lvert\mathrm{Sim}(\mathcal{A})\rvert$.
\end{prop}

\subsection{Categorical systole and volume}\label{sec3}

Let $(X,\omega,\Omega)$ be a compact Calabi--Yau manifold with a symplectic form $\omega$ and a holomorphic volume form $\Omega$.
The systole of $(X,\omega,\Omega)$ can be defined as
\begin{equation*}
\mathrm{sys}(X,\omega,\Omega) = \inf \left\{\left. \left\lvert \int_L \Omega \right\rvert \,\right|\, L \text{ is a special Lagrangian submanifold in } (X,\omega,\Omega) \right\}.
\end{equation*}
In view of this observation and the conjectural description of stability conditions on Fukaya categories \cite{Bri2,Joy} mentioned in the introduction, Fan \cite{Fan} introduced a categorical analogue of systole.

\begin{dfn}[\cite{Fan}]\label{dfn2.5}
The {\em categorical systole} of $\sigma = (Z,\mathcal{P}) \in \mathrm{Stab}(\mathcal{D})$ is defined by
\begin{equation*}
\mathrm{sys}(\sigma) = \inf\{\lvert Z(E) \rvert \,|\, E \text{ is a $\sigma$-stable object of } \mathcal{D}\}.
\end{equation*}
\end{dfn}

\begin{rmk}
By the condition Definition \ref{dfn2.1} (5), $\mathrm{sys}(\sigma) > 0$ for any $\sigma \in \mathrm{Stab}(\mathcal{D})$.
\end{rmk}

On the other hand, the volume of $(X,\omega,\Omega)$ is given by
\begin{equation*}
\mathrm{vol}(X,\omega,\Omega) = \left\lvert \int_X \Omega \wedge \overline{\Omega} \right\rvert.
\end{equation*}
For a basis $L_1,\dots,L_k$ of $H_d(X,\mathbb{Z})/\mathrm{Torsion}$ (where $d = \dim_\mathbb{C} X$), this can be rewritten as
\begin{equation*}
\mathrm{vol}(X,\omega,\Omega) = \left\lvert \sum_{i,j=1}^k \gamma^{ij} \int_{L_i} \Omega \int_{L_j} \overline{\Omega} \right\rvert
\end{equation*}
where $\gamma^{ij}$ is the $(i,j)$-component of the inverse matrix of the intersection matrix $(L_i \cdot L_j)_{i,j}$.
This leads to the following definition.

\begin{dfn}[\cite{FKY}]\label{dfn2.7}
Fix a basis $E_1,\dots,E_k$ of $\mathcal{N}(\mathcal{D})$.
The {\em categorical volume} of $\sigma = (Z,\mathcal{P}) \in \mathrm{Stab}(\mathcal{D})$ is defined by
\begin{equation*}
\mathrm{vol}(\sigma) = \left\lvert \sum_{i,j=1}^k \chi^{ij} Z(E_i) \overline{Z(E_j)} \right\rvert
\end{equation*}
where $\chi^{ij}$ is the $(i,j)$-component of the inverse matrix of the matrix $(\chi(E_i,E_j))_{i,j}$.
\end{dfn}

The following can be easily checked.

\begin{lem}[{\cite[Lemmas 2.7 and 2.12]{Fan}}]\label{lem2.8}
For $\Phi \in \mathrm{Auteq}(\mathcal{D})$, $\zeta \in \mathbb{C}$ and $\sigma \in \mathrm{Stab}(\mathcal{D})$,
\begin{itemize}
\item[(1)] $\mathrm{sys}(\Phi \cdot \sigma) = \mathrm{sys}(\sigma)$;
\item[(2)] $\mathrm{sys}(\sigma \cdot \zeta) = e^{\pi \mathrm{Im}(\zeta)} \mathrm{sys}(\sigma)$;
\item[(3)] $\mathrm{vol}(\Phi \cdot \sigma) = \mathrm{vol}(\sigma)$;
\item[(4)] $\mathrm{vol}(\sigma \cdot \zeta) = e^{2\pi \mathrm{Im}(\zeta)} \mathrm{vol}(\sigma)$.
\end{itemize}
\end{lem}

\begin{cor}\label{cor2.9}
The map $\mathrm{vol}/\mathrm{sys}^2$ defines a map from $\mathrm{Auteq}(\mathcal{D}) \backslash \mathrm{Stab}(\mathcal{D}) / \mathbb{C}$ to $[0,\infty)$.
\end{cor}

\section{Categorical systolic inequalities}

\subsection{Ginzburg dg algebras}

Let $Q$ be an ADE quiver, i.e., a quiver whose underlying graph is an ADE graph.
Let $Q_0 = \{1,\dots,n\}$ be the set of vertices of $Q$ and $Q_1$ be the set of arrows of $Q$.
The {\em 2-Calabi--Yau Ginzburg dg algebra} $\Gamma_Q = (\mathbb{K}\widehat{Q},d)$ associated to $Q$ is defined as follows \cite{Gin}.
First, $\mathbb{K}\widehat{Q}$ is the graded path algebra of the extended quiver $\widehat{Q}$ with vertices $\widehat{Q}_0 = Q_0$ with the following arrows:
\begin{itemize}
\item the original arrow $a \in Q_1$ (degree $0$);
\item the opposite arrow $a^* : j \to i$ for each $a : i \to j \in Q_1$ (degree $0$);
\item a loop $t_i : i \to i$ for each $i \in Q_0$ (degree $-1$).
\end{itemize}
The differential $d : \mathbb{K}\widehat{Q} \to \mathbb{K}\widehat{Q}$ is then defined by
\begin{itemize}
\item $da = da^* = 0$ for every $a \in Q_1$;
\item $dt_i = \sum_{a \in Q_1} e_i(aa^*-a^*a)e_i$ where $e_i$ denotes the constant path at $i \in Q_0$.
\end{itemize}

Let $D(\Gamma_Q)$ be the derived category of dg modules over $\Gamma_Q$.
The finite-dimensional derived category $\mathcal{D}_Q = D_\mathrm{fd}(\Gamma_Q)$ is defined to be the full triangulated subcategory of $D(\Gamma_Q)$ whose objects consist of dg modules $M$ such that $\mathrm{dim}\, H^*(M) < \infty$.

\begin{thm}[{\cite[Theorem 6.3]{Kel}}]
The category $\mathcal{D}_Q$ is 2-Calabi--Yau, i.e., there is an isomorphism
\begin{equation*}
\mathrm{Hom}(M,N) \cong \mathrm{Hom}(N,M[2])^\vee
\end{equation*}
which is functorial in both $M,N \in \mathcal{D}_Q$.
\end{thm}

There are $n$ simple dg modules $S_1,\dots,S_n$ corresponding to each of the $n$ vertices of $Q$.
It turns out that they generate $\mathcal{D}_Q$ and their configuration can be described as follows.

\begin{prop}[{\cite[Lemma 2.15]{KY}}]\label{prop3.2}
The category $\mathcal{D}_Q$ is generated by $S_1,\dots,S_n$ and
\begin{equation*}
\mathrm{Hom}^*(S_i,S_j) =
\begin{cases}
\mathbb{K} \oplus \mathbb{K}[-2] & (i=j)\\
\mathbb{K}[-1] & (i \sim j \text{ in } Q)\\
0 & (\text{otherwise})
\end{cases}
\end{equation*}
where $i \sim j$ in $Q$ if and only if $i$ is adjacent to $j$ in $Q$.
\end{prop}

Let $\mathcal{D}^{\leq 0}_\mathrm{can} \subset \mathcal{D}_Q$ be the full subcategory consisting of dg modules $M$ with $H^i(M) = 0$ for all $i > 0$.
This determines a bounded t-structure and its heart $\mathcal{A}_\mathrm{can}$ is of finite length and $\mathrm{Sim}(\mathcal{A}_\mathrm{can}) = \{S_1,\dots,S_n\}$.

Let us denote by $\mathrm{Stab}^\circ(\mathcal{D}_Q) \subset \mathrm{Stab}(\mathcal{D}_Q)$ the connected component containing $\mathrm{Stab}_{\mathcal{A}_\mathrm{can}}(\mathcal{D}_Q)$.
The following can be proved using \cite[Theorem 2.12]{Woo}.

\begin{prop}[{\cite[Corollary 5.3]{Qiu}}]\label{prop3.3}
For an ADE quiver $Q$,
\begin{equation*}
\mathrm{Stab}^\circ(\mathcal{D}_Q) = \bigcup_\mathcal{A} \mathrm{Stab}_\mathcal{A}(\mathcal{D}_Q)
\end{equation*}
where the union is over all hearts $\mathcal{A}$ obtained from $\mathcal{A}_\mathrm{can}$ by iterated simple forward/backward tilts.
\end{prop}

By Proposition \ref{prop3.2}, $S_1,\dots,S_n$ are {\em spherical objects} in the sense of \cite{ST}.
Thus each $S_i$ defines an exact autoequivalence $\Phi_i$ called the {\em spherical twist} which acts on $M \in \mathcal{D}_Q$ as
\begin{gather*}
\Phi_i(M) = \mathrm{Cone}(\mathrm{Hom}^*(S_i,M) \otimes S_i \to M),\\
\Phi_i^{-1}(M) = \mathrm{Cone}(M \to \mathrm{Hom}^*(M,S_i)^\vee \otimes S_i)[-1].
\end{gather*}
Let $\mathrm{Sph}(\mathcal{D}_Q)$ be the subgroup of $\mathrm{Auteq}(\mathcal{D}_Q)$ generated by $\Phi_1,\dots,\Phi_n$.

\begin{cor}\label{cor3.4}
For an ADE quiver $Q$,
\begin{equation*}
\mathrm{Stab}^\circ(\mathcal{D}_Q) = \bigcup_{\Phi \in \mathrm{Sph}(\mathcal{D}_Q)} \Phi \cdot \mathrm{Stab}_{\mathcal{A}_\mathrm{can}}(\mathcal{D}_Q).
\end{equation*}
\end{cor}

\begin{proof}
Let $\mathcal{A}^\sharp_i$ (resp. $\mathcal{A}^\flat_i$) be the simple forward tilt (resp. backward tilt) of $\mathcal{A}_\mathrm{can}$ with respect to $S_i$.
By Proposition \ref{prop3.2}, it follows that $\Phi^\sharp_{S_i}(M) = \Phi_i^{-1}(M)$ and $\Phi^\flat_{S_i}(M) = \Phi_i(M)$ for all $S_i \neq M \in \mathrm{Sim}(\mathcal{A}_\mathrm{can})$.
Since $\Phi_i(S_i) = S_i[-1]$, this implies that $\mathrm{Sim}(\mathcal{A}^\sharp_i) = \Phi_i^{-1}(\mathrm{Sim}(\mathcal{A}_\mathrm{can}))$ and $\mathrm{Sim}(\mathcal{A}^\flat_i) = \Phi_i(\mathrm{Sim}(\mathcal{A}_\mathrm{can}))$ by Proposition \ref{prop2.4}.
Therefore $\mathcal{A}^\sharp_i = \Phi_i^{-1}(\mathcal{A}_\mathrm{can})$, $\mathcal{A}^\flat_i = \Phi_i(\mathcal{A}_\mathrm{can})$ and
\begin{gather*}
\mathrm{Stab}_{\mathcal{A}^\sharp_i}(\mathcal{D}_Q)= \mathrm{Stab}_{\Phi_i^{-1}(\mathcal{A}_\mathrm{can})}(\mathcal{D}_Q) = \Phi_i^{-1} \cdot \mathrm{Stab}_{\mathcal{A}_\mathrm{can}}(\mathcal{D}_Q),\\
\mathrm{Stab}_{\mathcal{A}^\flat_i}(\mathcal{D}_Q)= \mathrm{Stab}_{\Phi_i(\mathcal{A}_\mathrm{can})}(\mathcal{D}_Q) = \Phi_i \cdot \mathrm{Stab}_{\mathcal{A}_\mathrm{can}}(\mathcal{D}_Q).
\end{gather*}
The assertion follows by iterating this process and applying Proposition \ref{prop3.3}.
\end{proof}

\begin{rmk}
If we consider the $d$-Calabi--Yau Ginzburg dg algebra $\Gamma_Q$ for general $d \geq 2$, we have $\mathcal{A}^{\sharp(d-1)}_i = \Phi_i^{-1}(\mathcal{A}_\mathrm{can})$ and $\mathcal{A}^{\flat(d-1)}_i = \Phi_i(\mathcal{A}_\mathrm{can})$ where $\mathcal{A}^{\sharp k}_i$ and $\mathcal{A}^{\flat k}_i$ are defined inductively by $\mathcal{A}^{\sharp k}_i = (\mathcal{A}^{\sharp (k-1)}_i)^\sharp_{S_i[k-1]}$ and $\mathcal{A}^{\flat k}_i = (\mathcal{A}^{\flat (k-1)}_i)^\flat_{S_i[-k+1]}$ (see \cite[Corollary 8.4]{KQ}).
Accordingly, the description of $\mathrm{Stab}^\circ(\mathcal{D}_Q)$ becomes more complicated when $d \geq 3$.
In this paper, we restrict our attention to the case $d=2$ to make computations manageable.
\end{rmk}

\subsection{Proof of Theorem \ref{thm1.1}}

By Corollaries \ref{cor2.9} and \ref{cor3.4}, it is enough to prove Theorem \ref{thm1.1} for the stability conditions in $\mathrm{Stab}_{\mathcal{A}_\mathrm{can}}(\mathcal{D}_Q) \subset \mathrm{Stab}^\circ(\mathcal{D}_Q)$.

Note that $K(\mathcal{D}_Q) \cong K(\mathcal{A}_\mathrm{can}) \cong \mathbb{Z}^n$ is generated by the classes of $S_1,\dots,S_n$.
Moreover, by Proposition \ref{prop3.2}, $(\chi(S_i,S_j))_{i,j}$ is the Cartan matrix of the underlying graph of $Q$.
In particular, it is non-degenerate so $\mathcal{N}(\mathcal{D}_Q) = K(\mathcal{D}_Q)$.

\begin{dfn}
Let $\Delta^+_Q$ be the subset of $K(\mathcal{D}_Q) \cong \mathbb{Z}^n$ consisting of those elements corresponding to the positive roots of the underlying graph of $Q$ so that $S_1,\dots,S_n$ correspond to the simple roots.
\end{dfn}

\begin{exa}
Let $Q$ be an $A_3$ quiver.
Then,
\begin{equation*}
\Delta^+_Q = \{S_1,S_2,S_3,S_1+S_2,S_2+S_3,S_1+S_2+S_3\} \subset K(\mathcal{D}_Q).
\end{equation*}
\end{exa}

\begin{prop}\label{prop3.8}
Let $Q$ be an ADE quiver with $n$ vertices.
Then, for any $\sigma = (Z,\mathcal{P}) \in \mathrm{Stab}_{\mathcal{A}_\mathrm{can}}(\mathcal{D}_Q)$,
\begin{equation}\label{eq3.1}
\mathrm{vol}(\sigma) = \sum_{i,j=1}^n \chi^{ij} Z(S_i) \overline{Z(S_j)} = \frac{1}{h_Q} \sum_{M \in \Delta^+_Q} \lvert Z(M) \rvert^2
\end{equation}
where $h_Q$ is the Coxeter number of the underlying graph of $Q$.
\end{prop}

\begin{proof}
It suffices to prove the second equality of \eqref{eq3.1}.
Indeed, it implies the middle of \eqref{eq3.1} is non-negative so the first equality follows from the definition of the categorical volume.

We will show the second equality of \eqref{eq3.1} by comparing the coefficients of $Z(S_i) \overline{Z(S_j)}$ in both sides for every $i \leq j$.
More specifically, we shall show that
\begin{equation}\label{eq3.2}
\chi^{ij} = \frac{1}{h_Q} \sum_{M \in \Delta^+_Q} c_i(M)c_j(M)
\end{equation}
for all $1 \leq i \leq j \leq n$, where $c_i(M) \in \mathbb{Z}_{\geq 0}$ denotes the coefficient of $S_i$ in $M$.

(1)
Let $Q$ be an $A_n$ quiver.
Label the vertices of $Q$ as follows (the orientations of the arrows are suppressed from the picture):
\begin{equation*}
\xymatrix@R=5mm@C=8mm{
\overset{1}{\bullet} \ar@{-}[r] & \overset{2}{\bullet} \ar@{-}[r] & \overset{3}{\bullet} \ar@{.}[r] & \overset{n-1}{\bullet} \ar@{-}[r] & \overset{n}{\bullet}
}
\end{equation*}
Then $\chi^{ij}$, which is the $(i,j)$-component of the inverse matrix of the Cartan matrix $(\chi(S_i,S_j))_{i,j}$, is given by
\begin{equation*}
\chi^{ij} = \min\{i,j\} - \frac{ij}{n+1}.
\end{equation*}
Moreover $\Delta^+_Q$ consists of $\frac{n(n+1)}{2}$ elements $R_{ij} = S_i + \cdots + S_j$ $(1 \leq i \leq j \leq n)$.

Now fix $i \leq j$.
Then $c_i(R_{kl})c_j(R_{kl}) \neq 0$ if and only if $c_i(R_{kl})c_j(R_{kl}) = 1$ if and only if $k \leq i \leq j \leq l$.
This implies that
\begin{equation*}
\frac{1}{h_Q} \sum_{M \in \Delta^+_Q} c_i(M)c_j(M) = \frac{1}{n+1} \cdot i(n+1-j) = i - \frac{ij}{n+1} = \chi^{ij}.
\end{equation*}

(2)
Let $Q$ be a $D_n$ quiver.
Label the vertices of $Q$ as follows:
\begin{equation*}
\xymatrix@R=5mm@C=8mm{
\overset{1}{\bullet} \ar@{-}[r] & \overset{2}{\bullet} \ar@{-}[r] & \overset{3}{\bullet} \ar@{.}[r] & \overset{n-2}{\bullet} \ar@{-}[r] \ar@{-}[d] & \overset{n-1}{\bullet}\\
&&& \underset{n}{\bullet} &
}
\end{equation*}
Then $\chi^{ij}$ for $i \leq j$ is given by
\begin{equation*}
\chi^{ij} =
\begin{cases}
i & (1 \leq i \leq j \leq n-2)\\
\frac{i}{2} & (1\leq i \leq n-2,j=n-1,n)\\
\frac{n-2}{4} & ((i,j)=(n-1,n))\\
\frac{n}{4} & ((i,j)=(n-1,n-1),(n,n))
\end{cases}
\end{equation*}
and those for $i > j$ can be obtained from this using $\chi^{ij} = \chi^{ji}$.
Moreover $\Delta^+_Q$ consists of $n(n-1)$ elements.
Concretely, $\Delta^+_Q$ consists of $\frac{(n-2)(n-1)}{2}$ elements of the form
\begin{equation*}
R_{ij} = S_i + \cdots + S_j \quad (1 \leq i \leq j \leq n-2),
\end{equation*}
$n-1$ elements of the form
\begin{equation*}
R^+_i = S_i + \cdots +S_{n-2} + S_{n-1} \quad (1 \leq i \leq n-1),
\end{equation*}
$n-1$ elements of the form
\begin{equation*}
R^-_i = S_i + \cdots +S_{n-2} + S_n \quad (1 \leq i \leq n-1)
\end{equation*}
(we set $R^-_{n-1} = S_n$), and $\frac{(n-2)(n-1)}{2}$ elements of the form
\begin{equation*}
R^a_i = S_i + \cdots + S_{n-a-2} + 2S_{n-a-1} + \cdots + 2S_{n-2} + S_{n-1} + S_n \quad (1 \leq i \leq n-a-2)
\end{equation*}
where $0 \leq a \leq n-3$.

\begin{itemize}

\item
Let $(i,j)=(n-1,n-1)$ (the case $(i,j) = (n,n)$ can be treated in the same way).
The elements of $\Delta^+_Q$ that contribute to the right hand side of \eqref{eq3.2} are $R^+_k,R^a_k$ and each of them contributes 1.
Therefore
\begin{equation*}
\frac{1}{h_Q} \sum_{M \in \Delta^+_Q} c_i(M)c_j(M) = \frac{1}{2(n-1)} \left[ (n-1) + \frac{(n-2)(n-1)}{2} \right] = \frac{n}{4} = \chi^{ij}.
\end{equation*}

\item
Next, let $(i,j) = (n-1,n)$.
In this case, the elements of $\Delta^+_Q$ that contribute to the right hand side of \eqref{eq3.2} are $R^a_k$ and each of them contributes 1.
Thus
\begin{equation*}
\frac{1}{h_Q} \sum_{M \in \Delta^+_Q} c_i(M)c_j(M) = \frac{1}{2(n-1)} \cdot \frac{(n-2)(n-1)}{2} = \frac{n-2}{4} = \chi^{ij}.
\end{equation*}

\item
Now fix $1\leq i \leq n-2,j=n-1$ (again, the case $j=n$ can be treated similarly).
The elements of $\Delta^+_Q$ that contribute to the right hand side of \eqref{eq3.2} are $R^+_k,R^a_k$ $(k \leq i)$.
There are $i$ such $R^+_k$ and each of them contributes 1.
On the other hand, if $0 \leq a \leq n-i-2$, each $R^a_k$ contributes 1 and the number of such $R^a_k$ is $(n-i-1)i$.
Moreover, if $n-i-1 \leq a \leq n-3$, each $R^a_k$ contributes 2 and the number of such $R^a_k$ is $\frac{(i-1)i}{2}$.
This shows that
\begin{equation*}
\frac{1}{h_Q} \sum_{M \in \Delta^+_Q} c_i(M)c_j(M) = \frac{1}{2(n-1)} \left[ i + (n-i-1)i + 2 \cdot \frac{(i-1)i}{2} \right] = \frac{i}{2} = \chi^{ij}.
\end{equation*}

\item
Finally, fix $1 \leq i \leq j \leq n-2$.
The elements of $\Delta^+_Q$ that contribute to the right hand side of \eqref{eq3.2} are $R_{kl},R^+_k,R^-_k,R^a_k$ $(k \leq i \leq j \leq l)$.
Among them, each $R_{kl}$ (resp. $R^\pm_k$) contributes 1 and there are $(n-j-1)i$ (resp. $i$) such elements.
On the other hand, each $R^a_k$ contributes 1 if $0 \leq a \leq n-j-2$ and there are $(n-j-1)i$ such $R^a_k$.
Moreover, if $n-j-1 \leq a \leq n-i-2$, each $R^a_k$ contributes 2 and the number of such $R^a_k$ is $(j-i)i$.
For the remaining case $n-i-1 \leq a \leq n-3$, each $R^a_k$ contributes 4 and there are $\frac{(i-1)i}{2}$ such $R^a_k$.
Consequently, we get
\begin{align*}
\frac{1}{h_Q}
&\sum_{M \in \Delta^+_Q} c_i(M)c_j(M)\\
&= \frac{1}{2(n-1)} \left[ (n-j-1)i + i + i + (n-j-1)i + 2 \cdot (j-i)i + 4 \cdot \frac{(i-1)i}{2} \right]\\
&= i = \chi^{ij}.
\end{align*}

\end{itemize}

(3)
The exceptional cases $E_6,E_7,E_8$ can be verified by direct computations.
\end{proof}

\begin{proof}[Proof of Theorem \ref{thm1.1}]
Let $\sigma = (Z,\mathcal{P}) \in \mathrm{Stab}_{\mathcal{A}_\mathrm{can}}(\mathcal{D}_Q)$.
Since $S_1,\dots,S_n$ are simple in $\mathcal{A}_\mathrm{can}$, they are stable for any $\sigma \in \mathrm{Stab}_{\mathcal{A}_\mathrm{can}}(\mathcal{D}_Q)$.
Therefore
\begin{equation*}
\mathrm{sys}(\sigma) = \inf\{\lvert Z(M) \rvert \,|\, M \text{ is a $\sigma$-stable object of } \mathcal{D}_Q\} \leq \inf_{1 \leq i \leq n} \lvert Z(S_i) \rvert.
\end{equation*}
Then, by Proposition \ref{prop3.8},
\begin{align*}
\mathrm{vol}(\sigma)
&= \frac{1}{h_Q} \sum_{M \in \Delta^+_Q} \lvert Z(M) \rvert^2 \geq \frac{1}{h_Q} \sum_{i=1}^n \lvert Z(S_i) \rvert^2\\
&\geq \frac{n}{h_Q} \inf_{1 \leq i \leq n} \lvert Z(S_i) \rvert^2 \geq \frac{n}{h_Q} \mathrm{sys}(\sigma)^2
\end{align*}
as desired.
\end{proof}

\section{Geometric viewpoint}\label{sec4}

\subsection{Milnor fibers}

Let $\mathcal{P}_n$ be the space of polynomials $z^{n+1} + a_1z^{n-1} + a_2z^{n-2} + \cdots + a_n \in \mathbb{C}[z]$ with only simple zeros.
This can be identified with the configuration space $\mathrm{Conf}_{n+1}^0(\mathbb{C})$ of $n+1$ points in $\mathbb{C}$ with the center of mass $0$ by sending $(z-\zeta_1) \cdots (z-\zeta_{n+1}) \in \mathcal{P}_n$ to $[\zeta_1,\dots,\zeta_{n+1}] \in \mathrm{Conf}_{n+1}^0(\mathbb{C})$.

For $p \in \mathcal{P}_n$, we consider the associated {\em Milnor fiber} of type $A_n$
\begin{equation*}
X_p = \{(x,y,z) \in \mathbb{C}^3 \,|\, x^2+y^2=p(z)\}
\end{equation*}
equipped with the symplectic form $\omega_p$ restricted from the standard symplectic form $\omega_\mathrm{std} = \frac{i}{2}(dx \wedge d\overline{x} + dy \wedge d\overline{y} + dz \wedge d\overline{z})$ on $\mathbb{C}^3$.
It is known that the symplectomorphism type of $(X_p,\omega_p)$ does not depend on the choice of $p \in \mathcal{P}_n$.
Indeed, all of them are symplectomorphic to the $A_n$-plumbing of the cotangent bundles of $S^2$.

Now consider the projection $\pi : X_p \to \mathbb{C};\, (x,y,z) \mapsto z$.
Let $\Delta_p \subset \mathbb{C}$ be the set of zeroes of $p$ and
\begin{equation*}
\Sigma_{p,\zeta} = \{(\sqrt{p(\zeta)}\cos\theta,\sqrt{p(\zeta)}\sin\theta,\zeta) \in \mathbb{C}^3 \,|\, \theta \in S^1 = \mathbb{R}/2\pi\mathbb{Z} \} \subset \pi^{-1}(\zeta)
\end{equation*}
(where $\sqrt{p}$ is a suitably chosen smooth square root of $p$).
Note that, for $\zeta \in \Delta_p$ (resp. $\zeta \in \mathbb{C} \setminus \Delta_p$), $\Sigma_{p,\zeta}$ is a point (resp. a circle).
A simple curve $\gamma : [0,1] \to \mathbb{C}$ such that $\gamma^{-1}(\Delta_p) = \{0,1\}$ and $\gamma(0) \neq \gamma(1)$ will be called a {\em matching path}.
For a matching path $\gamma$, we define the {\em matching cycle} associated to $\gamma$ by
\begin{equation*}
L_\gamma = \bigcup_{t \in [0,1]} \Sigma_{p,\gamma(t)}.
\end{equation*}
This is a Lagrangian sphere of $(X_p,\omega_p)$ and isotopic matching paths give Hamiltonian isotopic matching cycles \cite[Lemma 6.12]{KS}.

\begin{thm}[{\cite[Theorem 1]{Wu}}]\label{thm4.1}
Every exact Lagrangian submanifold in $(X_p,\omega_p)$ is Hamiltonian isotopic to a matching cycle.
\end{thm}

Note that a matching cycle is invariant under the $U(1)$-action on $X_p$ given by
\begin{equation*}
e^{i\theta} \cdot (x,y,z) = (x\cos\theta-y\sin\theta,x\sin\theta+y\cos\theta,z).
\end{equation*}
It turns out that the converse also holds.

\begin{lem}\label{lem4.2}
An exact Lagrangian submanifold in $(X_p,\omega_p)$ is $U(1)$-invariant if and only if it is a matching cycle.
\end{lem}

\begin{proof}
Every $U(1)$-orbit is contained in $\pi^{-1}(\zeta)$ for some $\zeta \in \mathbb{C}$.
For $\zeta \in \Delta_p$, a $U(1)$-orbit inside $\pi^{-1}(\zeta)$ is of the form $\{(re^{i\theta},ire^{i\theta},\zeta) \in \mathbb{C}^3 \,|\, \theta \in S^1 \}$ or $\{(re^{i\theta},-ire^{i\theta},\zeta) \in \mathbb{C}^3 \,|\, \theta \in S^1 \}$ for some $r \in \mathbb{R}_{\geq 0}$.
On the other hand, for $\zeta \in \mathbb{C} \setminus \Delta_p$, a $U(1)$-orbit inside $\pi^{-1}(\zeta)$ can be written as
\begin{equation*}
\left\{\left.\left(\frac{\sqrt{p(\zeta)}}{2}\left(re^{i\theta}+\frac{1}{r}e^{-i\theta}\right),\frac{\sqrt{p(\zeta)}}{2i}\left(re^{i\theta}-\frac{1}{r}e^{-i\theta}\right),\zeta\right) \in \mathbb{C}^3 \,\right|\, \theta \in S^1 \right\}
\end{equation*}
for some $r \in \mathbb{R}_{>0}$.
This can be seen from the $U(1)$-equivariant diffeomorphism
\begin{equation*}
\mathbb{C}^* \to \pi^{-1}(\zeta);\; u \mapsto \left(\frac{\sqrt{p(\zeta)}}{2}\left(u+\frac{1}{u}\right),\frac{\sqrt{p(\zeta)}}{2i}\left(u-\frac{1}{u}\right),\zeta\right)
\end{equation*}
where $U(1)$ acts on $\mathbb{C}^*$ by $e^{i\theta} \cdot u = e^{i\theta}u$.

Let $L$ be a $U(1)$-invariant exact Lagrangian submanifold in $(X_p,\omega_p)$ which is necessarily a $2$-sphere by Theorem \ref{thm4.1}.
Moreover, by the above description of the $U(1)$-orbits, the projected image $\pi(L)$ must be the image of a curve $\gamma : [0,1] \to \mathbb{C}$ such that $\gamma(0),\gamma(1) \in \Delta_p$, $\gamma(0) \neq \gamma(1)$ and $(0,0,\gamma(0)),(0,0,\gamma(1)) \in L$.
Take $0 < \varepsilon < 1$ so that $\gamma(t) \in \mathbb{C} \setminus \Delta_p$ for all $0 < t < \varepsilon$.
Near $(0,0,\gamma(0)) \in L$, let us parametrize $L$ by $\iota : (0,\varepsilon) \times S^1 \to X_p$ as follows:
\begin{align*}
\iota(t,\theta)
&= (x(t,\theta),y(t,\theta),\gamma(t))\\
&= \left(\frac{\sqrt{p(\gamma(t))}}{2}\left(r(t)e^{i\theta}+\frac{1}{r(t)}e^{-i\theta}\right),\frac{\sqrt{p(\gamma(t))}}{2i}\left(r(t)e^{i\theta}-\frac{1}{r(t)}e^{-i\theta}\right),\gamma(t)\right)
\end{align*}
for some smooth function $r : (0,\varepsilon) \to \mathbb{R}_{>0}$.
Below, we will see the Lagrangian condition $\iota^*\omega_p = 0$ implies that $\varepsilon$ can be taken to be $1$ and $r(t) = 1$ for all $0 < t < \varepsilon = 1$.
Then it follows that $\gamma$ is simple, $\gamma^{-1}(\Delta_p) = \{0,1\}$, $\iota(t,S^1) = \Sigma_{p,\gamma(t)}$ for all $t \in [0,1]$ and therefore $L = L_\gamma$.

By a direct computation, we have
\begin{align*}
\iota^*\omega_p\ = \frac{d}{dt} \left[ \frac{1}{4} \lvert p(\gamma(t)) \rvert \left( r(t)^2 - \frac{1}{r(t)^2} \right) \right] dt \wedge d\theta
\end{align*}
and so $\iota^*\omega_p = 0$ if and only if
\begin{equation*}
\frac{1}{4} \lvert p(\gamma(t)) \rvert \left( r(t)^2 - \frac{1}{r(t)^2} \right) = c
\end{equation*}
is a constant.
On the other hand,
\begin{equation*}
\frac{1}{4} \lvert p(\gamma(t)) \rvert \left( r(t)^2 - \frac{1}{r(t)^2} \right) = \mathrm{Im}(x(t,\theta)\overline{y(t,\theta)}) \to 0 \quad (t \to 0).
\end{equation*}
This shows that $c = 0$ and therefore $r(t) = 1$ for all $0 < t < \varepsilon$.
A similar argument shows that $\varepsilon$ can be taken to be $1$ (otherwise $L$ becomes an $A_k$-chain of $2$-spheres).
\end{proof}

\subsection{Special Lagrangian submanifolds}

We call $(X,\omega,\Omega)$ an {\em almost Calabi--Yau manifold} if $(X,\omega)$ is a (not necessarily compact) K\"ahler manifold and $\Omega$ is a holomorphic volume form.
Recall that a {\em special Lagrangian submanifold} of $(X,\omega,\Omega)$ is a Lagrangian submanifold $L$ of $(X,\omega)$ such that $\mathrm{Im}(e^{-i\phi}\Omega|_L) = 0$ for some $\phi \in \mathbb{R}$ called the {\em phase}.

As before, we define the {\em systole} of $(X,\omega,\Omega)$ by
\begin{equation*}
\mathrm{sys}(\Omega) = \mathrm{sys}(X,\omega,\Omega) = \inf \left\{\left. \left\lvert \int_L \Omega \right\rvert \,\right|\, L \text{ is a special Lagrangian submanifold in } (X,\omega,\Omega) \right\}.
\end{equation*}
Now, for simplicity, assume that the intersection product on $H_d(X,\mathbb{Z})/\mathrm{Torsion}$ (where $d = \dim_\mathbb{C} X$) is non-degenerate.
Fixing a basis $L_1,\dots,L_k$ of $H_d(X,\mathbb{Z})/\mathrm{Torsion}$, we define the {\em volume} of $(X,\omega,\Omega)$ by
\begin{equation*}
\mathrm{vol}(\Omega) = \mathrm{vol}(X,\omega,\Omega) = \left\lvert \sum_{i,j=1}^k \gamma^{ij} \int_{L_i} \Omega \int_{L_j} \overline{\Omega} \right\rvert
\end{equation*}
where $\gamma^{ij}$ is the $(i,j)$-component of the inverse matrix of the intersection matrix $(L_i \cdot L_j)_{i,j}$.

\begin{rmk}
If $X$ is compact, the volume $\mathrm{vol}(X,\omega,\Omega)$ coincides with the usual volume $\left\lvert \int_X \Omega \wedge \overline{\Omega} \right\rvert$.
In general, these two volumes $\mathrm{vol}(X,\omega,\Omega)$ and $\left\lvert \int_X \Omega \wedge \overline{\Omega} \right\rvert$ do not need to coincide for non-compact $X$.
It seems to be an interesting problem to study the relationship between $\mathrm{vol}(X,\omega,\Omega)$ and $\left\lvert \int_X \Omega \wedge \overline{\Omega} \right\rvert$.
\end{rmk}

From now on, let us specialize to the case of Milnor fibers.
As mentioned before, the Milnor fiber $(X_p,\omega_p)$ does not depend on the choice of $p \in \mathcal{P}_n$ as a symplectic manifold.
However, its complex structure depends on the choice of $p \in \mathcal{P}_n$.
More precisely, each $p \in \mathcal{P}_n$ determines a holomorphic volume form on $X_p$ by
\begin{equation*}
\Omega_p = \mathrm{Res}\, \frac{dx \wedge dy \wedge dz}{x^2 + y^2 - p(z)}
\end{equation*}
which also can be written as
\begin{equation*}
\Omega_p = \frac{dx \wedge dy|_{X_p}}{p'(z)} = - \frac{dy \wedge dz|_{X_p}}{2x} = \frac{dx \wedge dz|_{X_p}}{2y}.
\end{equation*}

\begin{lem}\label{lem4.4}
A special Lagrangian submanifold in $(X_p,\omega_p,\Omega_p)$ is $U(1)$-invariant.
\end{lem}

\begin{proof}
Let $L$ be a special Lagrangian submanifold in $(X_p,\omega_p,\Omega_p)$.
For every $t \in \mathbb{R}$, $L_t = e^{it} \cdot L$ is again a special Lagrangian submanifold (of the same phase).
On the other hand, the space of infinitesimal special Lagrangian deformations of $L$ can be identified with $H^1(L,\mathbb{R}) \cong H^1(S^2,\mathbb{R}) = 0$ \cite[Theorem 3.6]{McL}.
It implies that $L_t = L$ for all $t \in \mathbb{R}$.
\end{proof}

\begin{lem}\label{lem4.5}
A matching cycle $L_\gamma$ is a special Lagrangian submanifold in $(X_p,\omega_p,\Omega_p)$ if and only if $\gamma$ is a line segment.
\end{lem}

\begin{proof}
Consider the parametrization $\iota : (0,1) \times S^1 \to X_p$ of $L_\gamma$ given by
\begin{equation*}
\iota(t,\theta) = (\sqrt{p(\gamma(t))} \cos\theta,\sqrt{p(\gamma(t))} \sin\theta,\gamma(t)).
\end{equation*}
Then a direct computation shows that
\begin{equation*}
\iota^*\Omega_p = \frac{\gamma'(t)}{2} dt \wedge d\theta.
\end{equation*}
Thus $L_\gamma$ being a special Lagrangian submanifold means that $\arg \gamma'(t)$ is a constant, or equivalently that $\gamma$ is a line segment.
\end{proof}

Combining Lemmas \ref{lem4.2}, \ref{lem4.4} and \ref{lem4.5}, we obtain the following classification of special Lagrangian submanifolds in $(X_p,\omega_p,\Omega_p)$.

\begin{cor}\label{cor4.6}
Special Lagrangian submanifolds in $(X_p,\omega_p,\Omega_p)$ are precisely those of the forms $L_\gamma$ for line segments $\gamma$.
\end{cor}

\subsection{Proof of Theorem \ref{thm1.2}}

In this subsection, we will view $\mathcal{P}_n$ as the configuration space $\mathrm{Conf}_{n+1}^0(\mathbb{C})$.
Let $\mathcal{P}_n^\circ \subset \mathcal{P}_n$ be the configuration space of $n+1$ points in $\mathbb{C}$ in general position (with the center of mass $0$) in the sense that no 3 points of them lie on a single line.
For $p = [\zeta_1,\dots,\zeta_{n+1}] \in \mathcal{P}_n^\circ$, let us fix an order of $n+1$ points $\zeta_1,\dots,\zeta_{n+1}$.
We then define $\l_{ij}(p)$ ($1 \leq i \leq j \leq n$) to be the length of the line segment connecting $\zeta_i$ and $\zeta_{j+1}$.

Fix $p = [\zeta_1,\dots,\zeta_{n+1}] \in \mathcal{P}_n^\circ$ and let $L_{ij}$ ($1 \leq i \leq j \leq n$) be the matching cycle associated to the line segment connecting $\zeta_i$ and $\zeta_{j+1}$.
By Corollary \ref{cor4.6}, these exhaust all special Lagrangian submanifolds in $(X_p,\omega_p,\Omega_p)$.
Then since
\begin{equation*}
\left\lvert \int_{L_{ij}} \Omega_p \right\rvert = \pi \cdot l_{ij}(p),
\end{equation*}
we can write the systole of $(X_p,\omega_p,\Omega_p)$ as
\begin{equation*}
\mathrm{sys}(\Omega_p) = \pi \cdot \inf_{1 \leq i \leq j \leq n} l_{ij}(p).
\end{equation*}

Now take $L_1 = L_{11},\dots,L_n = L_{nn}$ as a basis of $H_2(X_p,\mathbb{Z}) \cong \mathbb{Z}^n$.
Their intersection matrix $(L_i \cdot L_j)_{i,j}$ is the Cartan matrix of type $A_n$ (under a suitable choice of orientations) \cite{KS}.
Then, as in the proof of Proposition \ref{prop3.8}, we can show that
\begin{equation*}
\mathrm{vol}(\Omega_p) = \frac{\pi^2}{n+1} \sum_{1 \leq i \leq j \leq n} l_{ij}(p)^2.
\end{equation*}

On the other hand, it is known that
\begin{equation}\label{eq4.1}
\mathrm{Sph}(\mathcal{D}_{A_n}) \backslash \mathrm{Stab}^\circ(\mathcal{D}_{A_n}) \simeq \mathcal{P}_n
\end{equation}
\cite[Theorem 6.4]{Tho} (also see \cite[Theorem 1.1]{Ike}).
For $p \in \mathcal{P}_n^\circ$, we can take a representative $\sigma_p = (Z,\mathcal{P}) \in \mathrm{Stab}^\circ(\mathcal{D}_{A_n})$ under this correspondence with the properties that there exists a $\sigma_p$-stable object $S_{ij}$ ($1 \leq i \leq j \leq n$) in the class $S_i + \cdots + S_j \in K(\mathcal{D}_{A_n})$ satisfying $Z(S_{ij}) = l_{ij}(p)$ and the set of $\sigma_p$-stable objects coincides with the set of shifts of $S_{ij}$ \cite{Tho}.
Note that, by Lemma \ref{lem2.8}, $\mathrm{sys}(\sigma) = \mathrm{sys}(\sigma_p)$, $\mathrm{vol}(\sigma) = \mathrm{vol}(\sigma_p)$ for any representative $\sigma \in \mathrm{Stab}^\circ(\mathcal{D}_{A_n})$ of the element corresponding to $p \in \mathcal{P}_n$ under the correspondence \eqref{eq4.1}.

\begin{prop}\label{prop4.7}
For every $p \in \mathcal{P}_n$,
\begin{gather*}
\mathrm{sys}(\Omega_p) = \pi \cdot \mathrm{sys}(\sigma_p),\\
\mathrm{vol}(\Omega_p) = \pi^2 \cdot \mathrm{vol}(\sigma_p).
\end{gather*}
\end{prop}

\begin{proof}
The above description of the representative $\sigma_p \in \mathrm{Stab}^\circ(\mathcal{D}_{A_n})$ shows that the assertion holds for every $p \in \mathcal{P}_n^\circ$.
As $\mathcal{P}_n^\circ$ is dense in $\mathcal{P}_n$, it also shows that the assertion holds for every $p \in \mathcal{P}_n$.
\end{proof}

\begin{proof}[Proof of Theorem \ref{thm1.2}]
Follows from Theorem \ref{thm1.1} and Proposition \ref{prop4.7}.
\end{proof}

\end{document}